\documentclass[a4paper, 12pt,reqno]{amsart}
\usepackage{amsmath,amssymb,amsthm,enumerate}
\allowdisplaybreaks
\theoremstyle{plain}
\newtheorem{theorem}{Theorem}
\newtheorem*{theorem*}{Theorem}
\newtheorem{lemma}{Lemma}
\newtheorem*{lemma*}{Lemma}
\theoremstyle{definition}

\newtheorem*{definition*}{Definition}
\theoremstyle{remark}
\newtheorem{remark}{Remark}
\newtheorem{example}{Example}
\newtheorem*{remark*}{Remark}

\newtheorem*{statement*}{Statement}

\begin{document}
\title[Certain functions, operators,  and related problems]{One class of functions with arguments in negative bases}

\author{Symon Serbenyuk}



\maketitle
\text{Kharkiv National University of Internal Affairs } \\
  \text{\emph{simon6@ukr.net}}

\begin{abstract}

The present research deals with generalizations of the Salem function with arguments defined in terms of certain alternating expansions of real numbers. The special attention is given to modelling such functions by systems of functional equations.

\end{abstract}

\section{Introduction}

Modelling and studying functions with with complicated local structure (singular (for example, \cite{{Salem1943}, {Zamfirescu1981}, {Minkowski}, {S.Serbenyuk 2017}}),  nowhere monotonic \cite{Symon2017, Symon2019}, and nowhere differentiable functions  (for example, see \cite{{Bush1952}, {Serbenyuk-2016}}, etc.) is one of the complex problems which were introduced  by  a number famous mathematicians.  For example, Banach, Bolzano, Bush, Darboux, Dini,  Dirichlet, Rieman,  Du Bois-Reymond, Hardy, Gerver, Minkowski, Weierstrass, Zamfrescu and other scientists investigated this problem. Brief historical remarks were given in~\cite{ACFS2017, {S. Serbenyuk systemy rivnyan 2-2}}. 

An interest in such functions is explained by their applications in  different areas of mathematics, as well as in physics, economics, technology, etc. for modelling real objects, processes, and phenomena  (for example, see~\cite{BK2000, ACFS2011, Kruppel2009, OSS1995, Sumi2009, Takayasu1984, TAS1993}).

  In \cite{Salem1943}, Salem introduced the following one of the simplest examples of singular functions:

$$
s(x)=s\left(\Delta^q _{i_1i_2...i_k...}\right)=\beta_{i_1}+ \sum^{\infty} _{k=2} {\left(\beta_{i_k}\prod^{k-1} _{l=1}{p_l}\right)}=y=\Delta^{P_q} _{i_1i_2...i_k...},
$$
where $1<q$ is a fixed positive integer, $p_j>0$,  and $\sum^{q-1} _{j=0}{p_j}=1$ for all $j=0,1, \dots , q-1$. This function is a singular function but its  generalizations  can be non-differentiable functions or those that do not have a derivative on a certain set. 
There is a numbers of  researches  devoted to the Salem function and its generalizations (for example, see \cite{ACFS2017, Kawamura2010, Symon2015, Symon2017, Symon2019, Symon2021} and references in these papers). 

The present paper extends the paper \cite{Symon2021} and is devoted to modelling and investigating generalizations of the Salem function by the generalized shifts in terms of the nega-$q$-ary representation (alternating expansions). The used techniques are more complicated for the case of alternating expansions.

Let us consider the basic notions and properties.

Let $1<q$ be a fixed positive integer,   $\Theta\equiv\{0,1,\dots ,q-1\}$ be an alphabet, and  $(i_k)$ be a sequence of numbers such that $i_k \in\Theta$ for all $k\in\mathbb N$. Then
\begin{equation}
\label{eq: nega-q-series}
\left[-\frac{q}{q+1}, \frac{1}{q+1}\right]\ni x=\Delta^{-q} _{i_1i_2...i_n...}\equiv\frac{i_1}{-q}+\frac{i_2}{(-q)^2}+\dots+\frac{i_k}{(-q)^k}+\dots .
\end{equation}
The last-mentioned expansion is called \emph{ a nega-$q$-ary expansion}, and the corresponding notation $\Delta^{-q} _{i_1i_2...i_n...}$ is \emph{the nega-$q$-ary representation of $x$}. The term ``nega" is used because the base of this numeral system is a negative number. Such expansions are partial case of real number expansions by alternating Cantor series (\cite{S. Serbenyuk alternating Cantor series 2013}).

Let us note that certain numbers from $\left[-\frac{q}{q+1}, \frac{1}{q+1}\right]$ have two different nega-$q$-ary representations of form \eqref{eq: nega-q-series}, i.e., 
$$
\Delta^{-q} _{i_1i_2\ldots i_{m-1}i_m[q-1]0[q-1]0[q-1]\ldots}=\Delta^{-q} _{i_1i_2\ldots i_{m-1}[i_m-1]0[q-1]0[q-1]\ldots}.
$$
Such numbers are called \emph{nega-$q$-rational}. The other numbers in $\left[-\frac{q}{q+1}, \frac{1}{q+1}\right]$  are called \emph{nega-$q$-irrational} and have the unique nega-$q$-ary representation.

Let $c_1,c_2,\dots, c_m$ be a fixed 
ordered tuple of integers such that $c_j\in\{0,1,\dots, q-~1\}$ for $j=\overline{1,m}$. 

\emph{A cylinder $\Lambda^{-q} _{c_1c_2...c_m}$ of rank $m$ with base $c_1c_2\ldots c_m$} is the following set 
$$
\Lambda^{-q} _{c_1c_2...c_m}\equiv\{x: x=\Delta^{-q} _{c_1c_2...c_m i_{m+1}i_{m+2}\ldots i_{m+k}\ldots}\}.
$$
That is,  any cylinder $\Lambda^{-q} _{c_1c_2...c_m}$ is a closed interval of the form:
$$
\left[\Delta^{-q} _{c_1c_2...c_m[q-1]0[q-1]0[q-1]...}, \Delta^{-q} _{c_1c_2...c_m0[q-1]0[q-1]0[q-1]...}\right] \text{whenever $m$ is even;}
$$
$$
\left[ \Delta^{-q} _{c_1c_2...c_m0[q-1]0[q-1]0[q-1]...}, \Delta^{-q} _{c_1c_2...c_m[q-1]0[q-1]0[q-1]...}\right] \text{whenever $m$ is odd.}
$$

Now, let us consider the notions of certain shifts. Such operators were desrcribed in \cite{Symon2021} for positive Cantor series. Here these operators will be considered for the case of nega-$q$-ary expansions.

 \emph{The shift operator $\sigma$  of expansion \eqref{eq:  nega-q-series}} is of the following form
$$
\sigma(x)=\sigma\left(\Delta^{-q} _{i_1i_2\ldots i_k\ldots}\right)=\sum^{\infty} _{k=2}{\frac{i_k}{(- q)^{k-1}}}=-q\Delta^{-q} _{0i_2\ldots i_k\ldots}=\Delta^{-q} _{i_2i_3i_4i_5i_6i_7\ldots i_k\ldots}.
$$
It is easy to see that 
\begin{equation*}
\begin{split}
\sigma^n(x) &=\sigma^n\left(\Delta^{-q} _{i_1i_2\ldots i_k\ldots}\right)\\
& =\sum^{\infty} _{k=n+1}{\frac{i_k}{(- q)^{k-n}}}=(-q)^n\Delta^{-q} _{\underbrace{0\ldots 0}_{n}i_{n+1}i_{n+2}\ldots}=\Delta^{-q} _{i_{n+1}i_{n+2}\ldots}.
\end{split}
\end{equation*}
Therefore, 
\begin{equation}
\label{eq: Cantor series 3}
x=\sum^{n} _{k=1}{\frac{i_k}{(-q)^k}}+\frac{1}{(-q)^n}\sigma^n(x).
\end{equation}

In \cite{S. Serbenyuk alternating Cantor series 2013} (see also \cite{{Symon2021}}), the notion of  the generalized shift operator was considered for Cantor series (mainly, for positive). 

Let $x$ be a number  is  represented by  expansion  \eqref{eq: nega-q-series}.\emph{The   generalized shift operator} is a map of the form
$$
\sigma_m(x)=\sum^{m-1} _{k=1}{\frac{i_k}{(- q)^k}}+\frac{i_{m+1}}{(-q)^m}+\sum^{\infty} _{l=m+2}{\frac{i_l }{(-q)^{l-1}}}.
$$
One can note that
\begin{equation}
\label{eq: generalized shift 1}
\sigma_m(x)=-qx-\frac{i_m}{(-q)^m}+(q+1) \Delta^{-q} _{i_1i_2...i_{m}00000...},
\end{equation}
where $\sigma_1=\sigma$.

The following remark gives auxiliary properties for modelling functions.
\begin{remark}
Suppose $x=\Delta^{-q} _{i_1i_2...i_k...}$ and $m$ is a fixed positive integer; then
$$
\sigma_{m}(x)=\sigma_{m}\left(\Delta^{-q} _{i_1i_2...i_k...}\right)=\Delta^{-q} _{i_1i_2...i_{m-1}i_{m+1}...},
$$
$$
\sigma_{m}\circ \sigma_{m}(x)=\sigma^{2} _{m}(x)=\sigma_m\left(\sigma_m(\Delta^{-q} _{i_1i_2...i_k...})\right)=\sigma_{m}\left(\Delta^{-q} _{i_1i_2...i_{m-1}i_{m+1}...}\right)=\Delta^{-q} _{i_1i_2...i_{m-1}i_{m+2}...},
$$
as well as  for  two positive integers $n_1$ and $n_2$, the following is true:
\begin{equation}
\label{eq: 2-composition}
\sigma_{n_2}\circ \sigma_{n_1}(x)=\sigma_{n_2}\left(\Delta^{-q} _{i_1i_2...i_{n_1-1}i_{n_1+1}...}\right)=\begin{cases}
\Delta^{-q} _{i_1i_2...i_{n_2-1}i_{n_2+1}...i_{n_1-1}i_{n_1+1}...}&\text{if $n_1>n_2$}\\
\Delta^{-q} _{i_1i_2...i_{n_1-1}i_{n_1+1}...i_{n_2-1}i_{n_2}i_{n_2+2}...}&\text{if $n_1<n_2$}\\
\Delta^{-q} _{i_1i_2...i_{m-1}i_{m+2}...}&\text{if $n_1=n_2=m$.}
\end{cases}
\end{equation}
For example,
$$
\sigma_5 \circ \sigma_5 (x)=\sigma^2 _5(x)=\sigma_5 \circ \sigma_5\left(\Delta^{-q} _{i_1i_2\ldots i_k\ldots}\right)=\sigma_5 \left(\Delta^{-q} _{i_1i_2i_3i_4i_6i_7i_8i_9\ldots}\right)=\Delta^{-q} _{i_1i_2i_3i_4i_7i_8i_9\ldots},
$$
$$
\sigma_1 \circ \sigma_3 (x)=\sigma_1 \circ \sigma_3\left(\Delta^{-q} _{i_1i_2\ldots i_k\ldots}\right)=\sigma_1 \left(\Delta^{-q} _{i_1i_2i_4i_5i_6i_7i_8i_9\ldots}\right)=\Delta^{-q} _{i_2i_4i_5i_6i_7i_8i_9\ldots},
$$
and
$$
\sigma_7 \circ \sigma_2 (x)=\sigma_7 \circ \sigma_2\left(\Delta^{-q} _{i_1i_2\ldots i_k\ldots}\right)=\sigma_7 \left(\Delta^{-q} _{i_1i_3i_4i_5i_6i_7i_8\ldots i_k\ldots}\right)=\Delta^{-q} _{i_1i_3i_4i_5i_6i_7i_9\ldots}.
$$
\end{remark}

\begin{remark}
Using the last remark, let us define an auxiliary sequence for modelling functions.
\label{rm: the main remark}
Suppose $(n_k)$ is a finite fixed sequence of positive integers such that $n_i\ne n_j$ for $i\ne j$.To delete the digits $i_{n_1}, i_{n_2}, \dots , i_{n_k}$ (according to this fixed order) by using a composition of the generalized shift operators in $x=\Delta^{-q} _{\alpha_1\alpha_2...\alpha_k...}$, we must consider auxiliary sequence, since the last remark is true.
That is, to construct
$$
x_0=\Delta^q _{\alpha_1\alpha_2...\alpha_{n_1-1}\alpha_{n_1+1}...\alpha_{n_2-1}\alpha_{n_2+1}...\alpha_{n_k-1}\alpha_{n_k+1}\alpha_{n_k+2}...\alpha_{n_k+t}...},~~~\mbox{where}~t=1,2,3, \dots ,
$$
let us define the  sequence $(\hat n_k)$, where
$$
\hat n_k= n_k-N_k,
$$
where $N_k$ is the number of all numbers which are less than $n_k$ in the finite fixed sequence $(n_1, n_2, \dots , n_k)$.
\end{remark}


\section{Modelling generalizations of the Salem function}

Nowadays it is well known that functional equations and systems of functional equations are widely used in mathematics and other sciences.   Modelling functions with complicated local structure  by systems of functional equations is a shining example of their applications in  function theory (\cite{Symon2019}). 

Suppose $(n_k)$ is a fixed sequence of positive integers such that $n_i\ne n_j$ for $i\ne j$ and such that  for any $n\in\mathbb N$ there exists a number $k_0$ for which the condition $n_{k_0}=n$ holds. Suppose $\hat n_k=n_k-N_k$ for all $k=1,2, 3, \dots$, where $N_k$ is the number of all numbers which are less than $n_k$ in the finite sequence $n_1, n_2, \dots , n_k$.

\begin{theorem}
Let $P_{-q}=(p_0,p_1,\dots , p_{q-1})$ be a fixed tuple of real numbers such that $p_i\in (-1,1)$, where $i=\overline{0,q-1}$, $\sum_i {p_i}=1$, and $0=\beta_0<\beta_i=\sum^{i-1} _{j=0}{p_j}<1$ for all $i\ne 0$. Then the following system of functional equations
\begin{equation}
\label{eq: system-q}
f\left(\sigma_{\hat n_{k-1}}\circ \sigma_{\hat n_{k-2}}\circ \ldots \circ \sigma_{\hat n_1}(x)\right)=\ddot\beta_{i_{n_k}}+\ddot p_{i_{n_k},}f\left(\sigma_{\hat n_{k}}\circ \sigma_{\hat n_{k-1}}\circ \ldots \circ \sigma_{\hat n_1}(x)\right),
\end{equation}
where $x=\Delta^{-q} _{i_1i_2...i_k...}$, $k=1,2, \dots$, and $\sigma_0(x)=x$, has the unique solution
$$
h(x)=\ddot\beta_{i_{n_1}}+\sum^{\infty} _{k=2}{\left(\ddot\beta_{i_{n_k}}\prod^{k-1} _{r=1}{\ddot p_{i_{n_r}}}\right)}
$$
in the class of determined and bounded on $\left[-\frac{q}{q+1}, \frac{1}{q+1}\right]$ functions. 

Here
$$
\ddot{\beta}_{i_{n_k}}=\begin{cases}
\beta_{i_{n_k}}&\text{if $n_k$ is   even }\\
\beta_{q-1-i_{n_k}}&\text{if $n_k$ is  odd}
\end{cases}
$$
and
$$
\ddot{p}_{i_{n_k}}=\begin{cases}
p_{i_{n_k}}&\text{if $n_k$  is even }\\
p_{q-1-i_{n_k}}&\text{if $n_k$  is odd.}
\end{cases}
$$
\end{theorem}
\begin{proof}
Since   $h$ is a function determined on $\left[-\frac{q}{q+1}, \frac{1}{q+1}\right]$, using system~\eqref{eq: system-q},  we obtain 
$$
h(x)=\ddot\beta_{i_{n_1}}+\ddot p_{i_{n_1}}g(\sigma_{\hat n_1}(x))
$$
$$
=\ddot \beta_{i_{n_1}}+\ddot p_{i_{n_1}}(\ddot \beta_{i_{n_2}}+\ddot p_{i_{n_2}}g(\sigma_{\hat n_2}\circ\sigma_{\hat n_1}(x)))=\dots
$$
$$
\dots =\ddot\beta_{i_{n_1}}+\ddot\beta_{i_{n_2}}\ddot p_{i_{n_1}}+\ddot\beta_{i_{n_3}}\ddot p_{i_{n_1}}\ddot p_{i_{n_2}}+\dots +\ddot \beta_{i_{n_k}}\prod^{k-1} _{r=1}{\ddot p_{i_{n_r}}}+\left(\prod^{k} _{t=1}{\ddot p_{i_{n_t}}}\right)h(\sigma_{\hat n_k}\circ \dots \circ \sigma_{\hat n_2}\circ \sigma_{\hat n_1}(x)).
$$

So,
$$
h(x)=\ddot\beta_{i_{n_1}}+\sum^{\infty} _{k=2}{\left(\ddot \beta_{i_{n_k}}\prod^{k-1} _{r=1}{\ddot p_{i_{n_r}}}\right)}
$$
since $h$ is a   function determined and bounded on the domain, as well as
$$
\lim_{k\to\infty}{h(\sigma_{\hat n_k}\circ \dots \circ \sigma_{\hat n_2}\circ \sigma_{\hat n_1}(x))\prod^{k} _{t=1}{\ddot p_{i_{n_t}}}}=0,
$$
where
$$
\prod^{k} _{t=1}{\ddot p_{i_{n_t}}}\le \left( \max_{0\le i\le q-1}{\ddot p_i}\right)^k\to 0, ~~~ k\to \infty.
$$
\end{proof}

\begin{example}
Suppose 
$$
(n_k)=(3, 7, 9, 5, 8, 12, 4, 6, 10, 11, 13, 14, 15, 16, 17, 18, \dots ).
$$
Then, using arguments explained in Remark~\ref{rm: the main remark}, we have the following:  $\hat n_1=n_1=3$,
$$
\hat n_2=n_2-1=7-1=6, \ \ \ \ \ \ \ \  \hat n_3=n_3-2=9-2=7, \ \ \ \  \ \ \ \ \hat n_4=n_4-1=5-1=4,
$$
$$
\hat n_5=n_5-3=5, \ \ \  \ \ \ \  \ \hat n_6=n_6-5=7, \ \ \ \ \ \ \  \ \ \ \ \ \ \ \ \ \ \  \  \hat n_7=n_7-1=3,
$$
$$
\hat n_8=n_8-3=3, \ \ \  \ \ \ \  \ \hat n_9=n_9-7=3, \ \ \ \ \ \ \  \ \ \ \ \ \ \ \ \ \ \ \ \  \  \hat n_{10}=n_{10}-8=3,
$$
$ \hat n_{10+k}=n_{10+k}- (n_{10+k}-3)=3$ for $k=1, 2, 3, \dots $.

So, we obtain the function
$$
h(x)=\ddot\beta_{i_{n_1}}+\sum^{\infty} _{k=2}{\left(\ddot\beta_{i_{n_k}}\prod^{k-1} _{r=1}{\ddot p_{i_{n_r}}}\right)}
$$
according to the sequence $(n_k)$. That is, for $x=\Delta^{-q} _{i_1i_2...i_k...}$, we have
$$
y=h(x)=\beta_{q-1-i_3(x)}+\beta_{q-1-i_7(x)}p_{q-1-i_3(x)}+\beta_{q-1-i_9(x)}p_{q-1-i_3(x)}p_{q-1-i_7(x)}
$$
$$
+\beta_{q-1-i_5(x)}p_{q-1-i_3(x)}p_{q-1-i_7(x)}p_{q-1-i_9(x)}+\beta_{i_8(x)}p_{q-1-i_3(x)}p_{q-1-i_7(x)}p_{q-1-i_9(x)}p_{q-1-i_5(x)}+\dots .
$$
\end{example}


\section{Properties of the main object of the research}


\begin{theorem} The following properties hold: 
\begin{enumerate}
\item The function $h$ is continuous at any nega-$q$-irrational point of $\left[-\frac{q}{q+1}, \frac{1}{q+1}\right]$.

\item The function $h$ is continuous at the nega-$q$-rational point
$$
x_0=\Delta^{-q} _{i_1i_2\ldots i_{m-1}i_m[q-1]0[q-1]0[q-1]\ldots}=\Delta^{-q} _{i_1i_2\ldots i_{m-1}[i_m-1]0[q-1]0[q-1]\ldots}
$$
whenever a sequence $(n_k)$ is such that the following conditions   hold:
\begin{itemize}
\item $k_0=\max\{k:  n_k \in \{1,2,\dots, m\}\}$, $n_{k_0}=m$, and $n_1, n_2, \dots , n_{k_0-1}\in \{1, 2, \dots , m-1\}$. 
\item the digit map $\psi: k \to n_k$ is such that  each even (odd) number $k$ assign only even (odd) number $n_k$  in all positions after  $n_{k_0}=m$.
\end{itemize}
Otherwise, the $q$-rational  point  $x_0$ is a point of discontinuity.

\item The set of all points of discontinuity of the function $g$ is a countable, finite, or empty set. It  depends on the sequence $(n_k)$.
\end{enumerate}
\end{theorem}
\begin{proof} One can begin with a remark on some notations.
\begin{remark}
For the compactness of notations, suppose the notation $\Delta^{h(x)} _{i_{n_1}i_{n_2}...i_{n_k}...}$ as the image of $\Delta^{-q} _{i_1i_2...i_k...}$ under the map $h$. Really, this can be written as following:
$$
h\left(\Delta^{-q} _{i_1i_2...i_k...}\right)=\Delta^{h(x)} _{\ddot i_{n_1}\ddot i_{n_2}...\ddot i_{n_{k}}...}, 
$$
where
$$
\ddot i_{n_k}=\begin{cases}
{i_{n_k}}&\text{if $n_k$  is even }\\
{q-1-i_{n_k}}&\text{if $n_k$  is odd.}
\end{cases}
$$
\end{remark}
One can remark  that any fixed function $h$ is given by a fixed sequence $(n_k)$ described above. One can write our  mapping by the following:
$$
g: x=\Delta^{-q} _{i_1i_2...i_k...}\to ~\ddot\beta_{i_{n_1}}+\sum^{\infty} _{k=2}{\left(\ddot\beta_{i_{n_k}}\prod^{k-1} _{r=1}{\ddot p_{i_{n_r}}}\right)}=\Delta^{h(x)} _{i_{n_1}i_{n_2}...i_{n_k}...}=h(x)=y.
$$

Let $x_0=\Delta^{-q} _{i_1i_2...i_k...}$ be an arbitrary nega-$q$-irrational number from $\left[-\frac{q}{q+1}, \frac{1}{q+1}\right]$. Let $x=\Delta^{-q} _{\gamma_1\gamma_2...\gamma_k...}$ be a  nega-$q$-irrational number such that the condition $\gamma_{n_j}=i_{n_j}$ holds  for all $j=\overline{1,k_0}$, where $k_0$ is a certain positive integer. That is, 
$$
x=\Delta^{-q} _{\gamma_1...\gamma_{n_1-1}i_{n_1}\gamma_{n_1+1}...\gamma_{n_2-1}i_{n_2}...\gamma_{(n_{(k_0-1)}+1)}...\gamma_{(n_{k_0}-1)}i_{n_{k_0}}\gamma_{n_{k_0}+1}...\gamma_{n_{k_0}+k}...}, ~k=1,2,\dots .
$$
Then
$$
h(x_0)=\Delta^{g(x)} _{i_{n_1}i_{n_2}...i_{n_{k_0}}i_{n_{k_0+1}}...},
$$
$$
h(x)=\Delta^{g(x)} _{i_{n_1}i_{n_2}...i_{n_{k_0}}\gamma_{n_{k_0+1}}...\gamma_{n_{k_0}+k}...}.
$$
Since $h$ is a bounded function, $ h(x) \le 1$, we get $h(x)-h(x_0)=$
$$
=\left(\prod^{k_0} _{j=1}{\ddot p_{i_{n_j}}}\right) \left(\ddot\beta_{\gamma_{n_{k_0+1}}}+\sum^{\infty} _{t=2}{\left(\ddot\beta_{\gamma_{n_{k_0+t}}}\prod^{k_0+t-1} _{r=k_0+1}{\ddot p_{\gamma_{n_r}}}\right)}-\ddot\beta_{i_{n_{k_0+1}}}-\sum^{\infty} _{t=2}{\left(\ddot\beta_{i_{n_{k_0+t}}}\prod^{k_0+t-1} _{r=k_0+1}{\ddot p_{i_{n_r}}}\right)}\right)
$$
$$
=\left(\prod^{k_0} _{j=1}{\ddot p_{i_{n_j}}}\right)\left(h(\sigma_{\hat n_{k_0}}\circ\ldots \sigma_{\hat n_2} \circ \sigma_{\hat n_1}(x))-h(\sigma_{\hat n_{k_0}}\circ\ldots \sigma_{\hat n_2} \circ \sigma_{\hat n_1}(x_0))\right),
$$
and
$$
|h(x)-h(x_0)|\le \delta\prod^{k_0} _{j=1}{\ddot p_{i_{n_j}}}\le \delta\left(\max\{p_0,\dots , p_{q-1}\}\right)^{k_0}\to 0 ~~~~~~~(k_0\to\infty). 
$$
Here $\delta$ is a certain real number.

So, $\lim_{x\to x_0}{h(x)}=h(x_0)$, i.e., the function $h$ is continuous at any nega-$q$-irrational point. 

Let $x_0$ be a nega-$q$-rational number, i.e.,
$$
x_0=x^{(1)} _0=\Delta^{-q} _{i_1i_2\ldots i_{m-1}i_m[q-1]0[q-1]0[q-1]\ldots}=\Delta^{-q} _{i_1i_2\ldots i_{m-1}[i_m-1]0[q-1]0[q-1]\ldots}=x^{(2)} _0.
$$
Then there exist positive integers $k^{*}$ and $k_0$ such that
$$
y_1=h\left(x^{(1)} _0\right)=\Delta^{g(x)} _{i_{n_1}i_{n_2}...i_{n_{k^{*}}}...i_{n_{k_0}}\breve\iota\breve\iota\breve\iota\breve\iota...},
$$
$$
y_2=h\left(x^{(2)} _0\right)=\Delta^{h(x)} _{i_{n_1}i_{n_2}...i_{n_{k^{*}-1}}[i_{n_{k^{*}}}-1]i_{n_{k^{*}+1}}...i_{n_{k_0}}\breve\iota\breve\iota\breve\iota\breve\iota...},
$$
where $\breve \iota\in \{0, q-1\}$.
 Here $n_{k^{*}}=m$, $n_{k^{*}}\le n_{k_0}$, and $k_0$ is a number such that $i_{n_{k_0}}\in\{i_1, \dots, i_{m-1}, i_m\}$ and ${k_0}$ is the maximum position of any number from  $\{1,2,\dots , m\}$ in the sequence $(n_k)$.

Since the representation $\Delta^{h(x)} _{i_{1}i_{2}...i_{k}...}$   is an  analytic representation of numbers for the case of positive $p_j$ (based on arguments from~\cite{Salem1943}, Section~2 in \cite{Symon2019},  and  \cite{Symon2021, Symon2017, Symon2015}), as well as usig definitions of $\ddot  \beta_{i_{n_k}}$ and $\ddot p_{i_{n_k}}$, we obtain that  for the digit map $\psi: k \to n_k$, that  each even (odd) number $k$ assign only even (odd) number $n_k$  in all positions after  $n_{k_0}=m$, the conditions hold:
$h(x^{(1)} _0)-h(x^{(2)} _0)=0.$

Since the Salem function is a strictly increasing function and using the case of a $q$-ary irrational number, let us consider the limits
$$
\lim_{x\to x_0+0}{h(x)}=\lim_{x\to x^{(1)} _0}{h(x)}=g(x^{(1)} _0)=y_1,~~~\lim_{x\to x_0-0}{h(x)}=\lim_{x\to x^{(2)} _0}{h(x)}=h(x^{(2)} _0)=y_2.
$$

Whence therem's conditions holds. The set of all points of discontinuity of the function $g$ is a countable, finite, or empty set. It  depends on the sequence $(n_k)$.
\end{proof}

Suppose $(n_k)$ is a fixed sequence and $c_{n_1}, c_{n_2}, \dots , c_{n_r}$ is a fixed tuple of numbers $c_{n_j}\in\{0,1,\dots , q-1\}$, where $j=\overline{1,r}$ and $r$ is a fixed positive integer.

Let us consider the following set
$$
\mathbb S_{-q, (c_{n_r})}\equiv \left\{x: x=\Delta^{-q} _{i_1i_2...i_{n_1-1}\overline{c_{n_1}}i_{n_1+1}...i_{n_2-1}\overline{c_{n_2}}...i_{n_{r}-1}\overline{c_{n_r}}i_{n_r+1}...i_{n_r+k}...}\right\},
$$
where $k=1,2,\dots $, and $\overline{c_{n_j}}\in \{c_{n_1}, c_{n_2}, \dots , c_{n_r}\}$ for all $j=\overline{1,r}$. This set has non-zero Lebesgue measure (for example, similar sets are investigated in  terms of other representations of numbers in~\cite{S. Serbenyuk alternating Cantor series 2013}). It is easy to see that 
$\mathbb S_{-q, (c_{n_r})}$ maps to
$$
h\left(\mathbb S_{-q, (c_{n_r})}\right)\equiv\left\{y: y=\Delta^{h(x)} _{c_{n_1} c_{n_2}\dots  c_{n_r}i_{n_{r+1}}...i_{n_{r+k}}...}\right\}
$$
under $g$.

For a  value $\mu_h \left(\mathbb S_{-q, (c_{n_r})}\right)$ of the increment, the following is true.
\begin{equation}
\label{eq: increment}
\mu_h \left(\mathbb S_{-q, (c_{n_r})}\right)=h\left(\sup\mathbb S_{-q, (c_{n_r})}\right)-h\left(\inf\mathbb S_{-q, (c_{n_r})}\right)=\mu_h \left(\left[\inf\mathbb S_{-q, (c_{n_r})}, \sup\mathbb S_{-q, (c_{n_r})}\right]\right)=\prod^{r} _{j=1}{\ddot p_{c_{n_j}}}.
\end{equation}
Let us note that one can consider the intervals $\left[\inf\mathbb S_{-q, (c_{n_r})}, \sup\mathbb S_{-q, (c_{n_r})}\right]$. Then $$
\sup\mathbb S_{-q, (c_{n_r})}-\inf\mathbb S_{-q, (c_{n_r})}=1-\sum^{r} _{j=1}{\frac{q-1}{q^{n_j}}}.
$$

So, one can formulate the following statements.
\begin{theorem}
The function $h$ has the following properties:
\begin{enumerate}
\item If $p_j\ge 0$ or $p_j>0$ for all $j=\overline{0,q-1}$, then:
\begin{itemize}
\item $h$ does not have intervals of monotonicity on $[0,1]$ whenever the condition $n_k=k$ holds for no  more than a finite number of values of $k$; 
\item $h$ has at least one interval of monotonicity on $[0,1]$ whenever  the condition $n_k\ne k$ holds for  a finite number of values of $k$; 
\item $g$ is a monotonic non-decreasing function (in the case when $p_j\ge 0$ for all $j=\overline{0,q-1}$) or is a strictly increasing function (in the case when $p_j> 0$ for all $j=\overline{0,q-1}$) whenever the condition $n_k=k$ holds for  $k\in\mathbb N$.
\end{itemize}
\item If there exists   $p_j=0$, where $j=\overline{0,q-1}$, then $g$ is  constant almost everywhere on $[0,1]$.
\item If there exists  $p_j<0$ (other $p_j$ are positive), where $j=\overline{0,q-1}$, and the condition $n_k=k$ holds for  almost all $k\in\mathbb N$, then $g$ does not have intervals of monotonicity on $[0,1]$.
\end{enumerate}
\end{theorem}

Let us note that  the last statements follow from~\eqref{eq: increment}.

\begin{lemma}
Let $\eta$ be a random variable  defined by the following form 
$$
\eta=\Delta^{q} _{\ddot\xi_{1}\ddot\xi_{2}...\ddot\xi_{k}...},
$$
where
$$
\ddot\xi_k=\begin{cases}
i_k &\text{if $k$  is even }\\
{q-1-i_k}&\text{if $k$  is odd},
\end{cases}
$$
 $k=1,2,3,\dots $, and the digits $\xi_{k}$ are  random and take the values $0,1,\dots ,q-1$ with probabilities ${p}_{0}, {p}_{1}, \dots , {p}_{q-1}$.
That is,  $\xi_n$ are independent and $P\{\xi_{k}=i_{n_k}\}=p_{i_{n_k}}$, $i_{n_k}\in\{0,1,\dots q-1\}$. 
 Here $(n_k)$ is a sequence of positive integers such that $n_i\ne n_j$ for $i\ne j$ and such that  for any $n\in\mathbb N$ there exists a number $k_0$ for which the condition $n_{k_0}=n$ holds.

The distribution function $\ddot{F}_{\eta}$ of the random variable $\eta$ can be
represented by
$$
\ddot{F}_{\eta}(x)=\begin{cases}
0,&\text{ $x< 0$}\\
\ddot\beta_{i_{n_1}(x)}+\sum^{\infty} _{k=2} {\left({\ddot\beta}_{i_{n_k}(x)} \prod^{k-1} _{r=1} {\ddot{p}_{i_{n_r}(x)}}\right)},&\text{ $0 \le x<1$}\\
1,&\text{ $x\ge 1$.}
\end{cases}
$$
\end{lemma}

A method of  the corresponding proof is described in~\cite{Symon2017}.

\begin{theorem}
The Lebesgue integral of the function $h$ can be calculated by the
formula
$$
\int^{\frac{1}{q+1}} _{-\frac{q}{q+1}} {h(x)dx}=\frac{1}{q+1}\sum^{q-1} _{j=0}{\ddot\beta_j}.
$$
\end{theorem}
\begin{proof}

By $\ddot A$ denote the sum $\frac{1}{q}\sum^{q-1} _{j=0}{\ddot\beta_j}$ and by $\ddot B$ denote the sum $\frac{1}{q}\sum^{q-1} _{j=0}{\ddot p_j}$. Since equality~\eqref{eq: generalized shift 1} holds, we obtain
$$
x=-\frac{1}{q}\sigma_m(x)+\frac{q+1}{q}\sum^{m} _{k=1}{\frac{i_k}{(-q)^k}}+\frac{i_m}{(-q)^{m+1}}
$$
and
$$
dx=-\frac{1}{q}d(\sigma_m(x)).
$$

In the general case, for arbitrary positive integers $n_1$ and $n_2$, using equality \eqref{eq: 2-composition}, we have 
$$
\sigma_{\bar n_2}\circ \sigma_{\bar n_1}(x)=\begin{cases}
\Delta^{q} _{i_1i_2...i_{n_2-1}i_{n_2+1}...i_{n_1-1}i_{n_1+1}...}&\text{whenever $n_1>n_2$}\\
\Delta^{q} _{i_1i_2...i_{n_1-1}i_{n_1+1}...i_{n_2-1}i_{n_2+1}i_{n_2+2}...}&\text{whenever  $n_1<n_2$}
\end{cases}
$$
and
$$
d\left(\sigma_{\bar n_{k-1}}\circ \dots \circ \sigma_{\bar n_2}\circ \sigma_{\bar n_1}(x)\right)=-\frac{1}{q}d\left(\sigma_{\bar n_k}\circ \sigma_{\bar n_{k-1}}\circ \dots \circ \sigma_{\bar n_2}\circ \sigma_{\bar n_1}(x)\right), 
$$
where $k\in\mathbb N$ and $\sigma_0(x)=x$.

So, we have
$$
\int^{\frac{1}{q+1}} _{-\frac{q}{q+1}} {h(x)dx}=\sum^{q-1} _{j=0}{\int^{\sup{\Lambda^{-q} _{j}} } _{\inf{\Lambda^{-q} _{j}}} {h(x)dx}}=\sum^{q-1} _{j=0}\int^{\sup{\Lambda^{-q} _{j}} } _{\inf{\Lambda^{-q} _{j}}}{{\left(\ddot\beta_j+\ddot p_jh(\sigma_{\bar n_1}(x))\right)dx}}
$$
$$
=\frac{1}{q}\sum^{q-1} _{j=0}{\ddot\beta_j}-\frac{1}{q}\left(\sum^{q-1} _{j=0}{\ddot p_j}\right)\int^{\frac{1}{q+1}} _{-\frac{q}{q+1}} {h(\sigma_{\bar n_1}(x))d(\sigma_{\bar n_1}(x))}
$$
$$
=\frac{1}{q}\sum^{q-1} _{j=0}{\ddot\beta_j}-\frac{1}{q}\left(\sum^{q-1} _{j=0}{\ddot p_j}\right)\left(\sum^{q-1} _{j=0}{\int^{\sup{\Lambda^{-q} _{j}} } _{\inf{\Lambda^{-q} _{j}}} {\left(\ddot \beta_j+\ddot p_jh(\sigma_{\bar n_2}\circ \sigma_{\bar n_1}(x))\right)d(\sigma_{\bar n_1}(x))}}\right)
$$
$$
=\ddot A-\ddot B\left(\ddot A-\ddot B\int^{\frac{1}{q+1}} _{-\frac{q}{q+1}} {h(\sigma_{\bar n_2}\circ \sigma_{\bar n_1}(x)))d(\sigma_{\bar n_2}\circ \sigma_{\bar n_1}(x))}\right)
$$
$$
=\ddot A-\ddot A\ddot B+\ddot B^2\left(\sum^{q-1} _{j=0}{\int^{\sup{\Lambda^{-q} _{j}} } _{\inf{\Lambda^{-q} _{j}}} {\left(\ddot\beta_j+\ddot p_jh(\sigma_{\bar n_3}\circ\sigma_{\bar n_2}\circ \sigma_{\bar n_1}(x))\right)d(\sigma_{\bar n_2}\circ\sigma_{\bar n_1}(x))}}\right)
$$
$$
=\ddot A-\ddot A\ddot B+\ddot B^2\left(\ddot A-\ddot B\int^{\frac{1}{q+1}} _{-\frac{q}{q+1}} {h(\sigma_{\bar n_3}\circ\sigma_{\bar n_2}\circ \sigma_{\bar n_1}(x)))d(\sigma_{\bar n_3}\circ\sigma_{\bar n_2}\circ \sigma_{\bar n_1}(x))}\right)
$$
$$
=\ddot A-\ddot A\ddot B+\ddot A\ddot B^2-\ddot B^3\left(\ddot A-\ddot B\int^{\frac{1}{q+1}} _{-\frac{q}{q+1}} {h(\sigma_{\bar n_4}\circ\sigma_{\bar n_3}\circ\sigma_{\bar n_2}\circ \sigma_{\bar n_1}(x)))d(\sigma_{\bar n_4}\circ\sigma_{\bar n_3}\circ\sigma_{\bar n_2}\circ \sigma_{\bar n_1}(x))}\right)=\dots
$$
$$
\dots = \ddot A+\ddot A\ddot B+\dots + A\ddot B^{2t-2}-\ddot A\ddot B^{2t-1}
$$
$$
+\ddot B^k\left(\ddot A-\ddot B\int^{\frac{1}{q+1}} _{-\frac{q}{q+1}} {h(\sigma_{\bar n_{k+1}}\circ\sigma_{\bar n_k}\circ\ldots \circ \sigma_{\bar n_1}(x)))d(\sigma_{\bar n_{k+1}}\circ\sigma_{\bar n_k}\circ \ldots \circ \sigma_{\bar n_1}(x))}\right).
$$

Since
$$
\sum^{q-1} _{j=0}{\ddot p_j}=1, ~~~\ddot B^{k+1}=\left(\frac{1}{q}\sum^{q-1} _{j=0}{\ddot p_j}\right)^{k+1}=\left(\frac{1}{q}\right)^{k+1} \to 0 ~\text{as}~ k\to\infty,
$$
we obtain
$$
\int^{\frac{1}{q+1}} _{-\frac{q}{q+1}} {h(x)dx}=\lim_{k\to\infty}{\left(\sum^{k} _{t=0}{\ddot A (-\ddot B)^t}-\ddot B^{k+1}\int^{\frac{1}{q+1}} _{-\frac{q}{q+1}}  {h(\sigma_{\bar n_{k+1}}\circ\sigma_{\bar n_k}\circ\ldots \circ \sigma_{\bar n_1}(x)))d(\sigma_{\bar n_{k+1}}\circ\sigma_{\bar n_k}\circ \ldots \circ \sigma_{\bar n_1}(x))}\right)}
$$
$$
=\sum^{\infty} _{k=0}{\ddot A(-\ddot B)^k}=\left(\frac{\sum^{q-1} _{j=0}{\ddot \beta_j}}{q}\right)\left(\sum^{\infty} _{k=0}{\frac{1}{(-q)^{k}}}\right)=\frac{1}{q+1}\sum^{q-1} _{j=0}{\ddot \beta_j}.
$$
\end{proof}

Since now researchers are trying to find simpler examples of singular  functions, differential properties will be investigated in the next papers of the author of this article.

\end{document}